\documentclass{amsart}
\usepackage{amsmath}
\usepackage{mathscinet}
\usepackage{amsfonts}
\usepackage{amsthm}
\usepackage{amssymb}
\usepackage{mathrsfs}
\usepackage{color}
\usepackage{tkz-graph}

\usepackage{ulem}

\usepackage{listings}

\usepackage{hyperref}

\usepackage{enumitem}

\newtheorem{theorem}{Theorem}[section]
\newtheorem{lemma}[theorem]{Lemma}

\newtheorem{problem}[theorem]{Problem}
\newtheorem{cor}[theorem]{Corollary}

\newtheorem{fact}[theorem]{Fact}

\newtheorem{example}[theorem]{Example}

\theoremstyle{definition}

\theoremstyle{remark}
\newtheorem{rem}[theorem]{Remark}

\numberwithin{equation}{section}
\numberwithin{theorem}{section}
\numberwithin{figure}{section}

\def\N{{\mathbf{N}}}

\def\f2{\mathbb{F}_2}

\newcommand{\diam}{{\rm diam}\hskip0.02cm}

\DeclareMathOperator{\Lam}{\mathsf{La}}

\newcommand{\bbN}{\mathbb{N}}

\newcommand{\de}{\delta}
\newcommand{\e}{\varepsilon}

\newcommand{\la}{\lambda}

\newcommand{\ro}{\varrho}

\newcommand{\vf}{\varphi}

\newcommand{\De}{\Delta}
\newcommand{\bbZ}{\mathbb{Z}}

\newcommand\remove[1]{}

\newcommand{\Om}{\Omega}
\newcommand{\La}{\Lambda}

\newcommand{\lb}{\label}

\newcommand{\wtw}{if and only if}

\newcommand{\WB}{\mathsf{WB}}

\def \B {\mathcal{B}}

\def \R {\mathbb{R}}

\def \E {\mathbb{E}}

\def \N {\mathbb{N}}
\def \P {\mathcal{P}}

\def \U {\mathcal{U}}

\def \Z {\mathbb{Z}}

\def \eps {\varepsilon}

\def \sbs {\subseteq}
\def \sps {\supseteq}

\DeclareMathOperator{\Lip}{Lip}
\DeclareMathOperator{\TSP}{TSP}

\DeclareMathOperator{\TSCP}{TSCP}

\title{Nonlinear type and metric embeddings of lamplighter spaces}

\author{C. Gartland}\thanks{The first named author was supported by NSF Award DMS-2555144}

\address{Department of Mathematics and Statistics, University of North Carolina at Charlotte, University
City Blvd, Charlotte, NC 28223, USA}

\author{B. Randrianantoanina}\thanks{The second named author was supported by NSF under Grant No. 2620717}

\address{Department of Mathematics, Miami University, Oxford, OH
45056, USA}

\email{randrib@miamioh.edu} 

\author{N. L. Randrianarivony}
  
\address{Department of Mathematics and Statistics, Saint Louis University, St. Louis, MO 63103, USA}

\email{nirina.randrianarivony@slu.edu} 


\begin{document}

\maketitle

\begin{center}
{\it Dedicated to William B. Johnson on the occasion of his 80th birthday.}
\end{center}

\begin{abstract}
We prove that for all metric spaces $X$ the following properties of the lamplighter space $\mathsf{La}(X)$ are equivalent: 
(1) every snowflake of $\mathsf{La}(X)$ admits a biLipschitz embedding into a finite product of $\R$-trees, (2) every snowflake of $\mathsf{La}(X)$ admits a biLipschitz embedding into a Hilbert space, (3) $\mathsf{La}(X)$ has finite Nagata dimension, (4) $\mathsf{La}(X)$ has Markov type 2, (5) $\mathsf{La}(X)$ has nontrivial Enflo type, (6) $\mathsf{La}(X)$ does not contain the Hamming cubes with uniform distortion. We characterize metric spaces $X$ for which $\mathsf{La}(X)$ satisfies properties (1)-(6) as those that are ``$\TSP$-efficient" - a new condition that we introduce - which roughly means that the traveling salesman problem in $X$ can be solved as ``efficiently" as the traveling salesman problem in $\mathbb{R}$, up to a constant multiplicative factor. We also prove that if such metric spaces $X$ admit a biLipschitz embedding into $\mathbb{R}^n$, then $\mathsf{La}(X)$ admits a biLipschitz embedding into a finite product of $\mathbb{R}$-trees, and therefore into every nonsuperreflexive Banach space. Finally, we give a full characterization of metric spaces $X$ such that the lamplighter space $\mathsf{La}(X)$ biLipschitz embeds into a Hilbert space.
\end{abstract}

\tableofcontents

\section{Introduction}
Lamplighter spaces over metric spaces are natural generalizations of wreath products of groups and have been introduced and studied by several authors, cf. \cite{E2006,Donno,BMSZ21,GT24}. See \S\ref{sec:lamplighter} for precise definition and Remark~\ref{rem:metric} for the connection between the wreath products $\Z_2 \wr G$ and the lamplighter space over $G$. The geometry of lamplighter spaces is intimately connected to the traveling salesman problem, see \cite{NP11} for an insightful discussion of these connections, where Naor and Peres posed an important problem whether the lamplighter group $\bbZ_2\wr (\bbZ\times\bbZ)$ admits a biLipschitz embedding into $L_1$. This motivated subsequent study of biLipschitz embeddings of lamplighter spaces into $L_1$ (see e.g. \cite{BMSZ21,BMSZ22}) and also our work. We remark that the recent preprint \cite{planarTSP} shows that $\Z_2 \wr (\Z\times\Z)$ admits no biLipschitz embedding into a metric space of negative-type, which resolves the problem of Naor-Peres since $L_1$ is of negative-type.

Our main result identifies a necessary and sufficient condition for a metric space $X$ such that the lamplighter space $\Lam(X)$ has Markov type 2, which turns out to also be equivalent to the property that every snowflake of $\Lam(X)$ biLipschitz embeds into a finite product of $\R$-trees. We call this condition {\it TSP-efficiency}. The precise definition is given in \S\ref{sec:TSPefficient}, but, intuitively, a metric space $X$ is $\TSP$-efficient if it is `path-like' in a sense that for every subset of $X$ there exists a route passing through every point of the set and of length proportional to the diameter of the set, and therefore the traveling salesman problem in $X$ can be solved as efficiently as the traveling salesman problem in $\mathbb{R}$, up to a constant multiplicative factor. The following is our main result. Recall that a {\it snowflake} of a semimetric space $(Y,\rho)$ is the space $(Y,\rho^{1-\delta})$ for some $\delta\in(0,1)$. See \S\ref{sec:prelim} for the other necessary definitions.

\begin{theorem} [Theorem~\ref{thm:main}] \label{thm:mainintro}
Let $X$ be a metric space. Then the following are equivalent.
\begin{enumerate}[label=(\roman*)]
    \item $X$ is $\TSP$-efficient.
    
    \item Every snowflake of $\Lam(X)$ biLipschitz embeds into a finite product of $\R$-trees.
   
    \item Every snowflake of $\Lam(X)$ biLipschitz embeds into a Hilbert space.
    
    \item $\Lam(X)$ has finite Nagata dimension.
    
    \item $\Lam(X)$ has Markov type 2.
   
    \item $\Lam(X)$ has nontrivial Enflo type.
   
    \item $\Lam(X)$ does not contain the Hamming cubes with uniformly bounded biLipschitz distortion.
\end{enumerate}
\end{theorem}

\begin{rem}
Of course, any property of $(\Lam(X),\TSP)$ between Markov type 2 and noncontainment of the Hamming cubes can be added to the list of equivalences in Theorem~\ref{thm:mainintro}, such as nontrivial Markov type or  Enflo type $2$.
\end{rem}

The equivalence $(i)\Leftrightarrow  (vii)$ in Theorem~\ref{thm:mainintro} fully   resolves a question of Baudier, Motakis, Schlumprecht, and Zs\'ak \cite{BMSZ21},  who asked  what condition on a family of metric spaces $\{X_n\}_n$ implies that the Hamming cubes embed into $\{\Lam(X_n)\}_n$ with uniformly bounded distortion. In  \cite{BMSZ21} they showed that a sufficient condition is that complete graphs $\{K_k\}_k$ embed into $\{X_n\}_n$ with uniformly bounded distortion, and suggested this may not be a necessary condition, which was confirmed in \cite{MR26}. 
 
In particular, it follows from our characterization that the lamplighter group $\Lam(\Z^2)$ contains the Hamming cubes with uniformly bounded biLipschitz distortion, which, to the best of our knowledge, is a new result. Indeed, we show 
that every $\TSP$-efficient metric space has Assouad dimension $\leq 1$ (Lemma~\ref{lem:TSPdoubling}), and thus  
$\Z^2$ is not $\TSP$-efficient  (Remark~\ref{rem:Zsquare}).

The equivalence of $(i)\Leftrightarrow (iv)$ in Theorem~\ref{thm:mainintro} can be seen as a generalization to the metric space setting of a result of Brodskiy, Dydak, and Lang \cite{BDL14} who proved that if a group $H\ne 1$ is finite and a group $G$ is finitely generated then the Nagata dimension of the wreath product $H\wr G$ is finite if and only if the growth of $G$ is bounded by a linear function, and moreover if $G$ is infinite with linear growth, then the Nagata dimension of $H\wr G$ is 1.

It is natural to ask whether item $(ii)$ in Theorem~\ref{thm:mainintro} can be replaced by a stronger condition that $\Lam(X)$ itself, and not just its snowflakes, biLipschitz embeds into a finite product of $\R$-trees. We show the following sufficient condition:

\begin{theorem}  [see Corollary~\ref{cor:induced1-2} for  precise statement]  \label{cor:induced-intro} 
 If a metric space $X$ is 
$\TSP$-efficient and admits a biLipschitz embedding into a finite dimensional space $\R^m$, then $\Lam(X)$ biLipschitz embeds into a finite product of $\R$-trees, and thus into every non-superreflexive Banach space.
\end{theorem} 

This suggests the following problem:

\begin{problem}\lb{problem-TSP-fd}
Is it true that every $\TSP$-efficient metric space $X$ bilipschitz embeds into a finite dimensional space $\R^m$, where $m$ and the biLipschitz distortion depend only on the $\TSP$-efficiency constant of $X$?
\end{problem}

One piece of evidence in favor of a positive answer towards Problem~\ref{problem-TSP-fd} is that every $\TSP$-efficient metric space has sharp Assouad dimension $\leq 1$ (Lemma~\ref{lem:TSPdoubling}), and at present, no example of a space with sharp Assouad dimension $\leq 1$ failing to embed into a finite dimensional Euclidean space is known.

It is known that the answer to Problem~\ref{problem-TSP-fd} is affirmative in the case when $X$ is a tree. Thus, as a corollary of Theorem~\ref{cor:induced-intro} and known results about biLipschitz embeddability of doubling trees into finite dimensional spaces, we obtain that lamplighter spaces on $\TSP$-efficient trees admit biLipschitz embeddings into finite products of $\R$-trees, and thus also into every non-superreflexive Banach space (see Corollary~\ref{thm:trees}). Theorem~\ref{cor:induced-intro} also gives a new short proof of \cite[Theorem~1.3]{OR19} and \cite[Theorem~1.2]{BMSZ21} (see \S\ref{sec4}).

A natural question that arises from Theorem~\ref{thm:mainintro} is under which conditions does $\Lam(X)$ itself, and not just its snowflakes, biLipschitz embed into a Hilbert space? Our next theorem provides a complete classification of such spaces. Recall that a metric space has Nagata dimension 0 \wtw\ it is biLipschitz equivalent to an ultrametric space (see \S\ref{sec:prelim} and \S\ref{sec:ultra} for necessary definitions).

\begin{theorem} [Theorem~\ref{thm:main2}] \label{thm:main2intro}
Let $X$ be a metric space. Then the following are equivalent.
\begin{enumerate}[label=(\roman*)]
    \item $X$ is $\TSP$-efficient and has Nagata dimension 0.

    \item $X$ has Assouad dimension $<1$.
   
    \item $X$ biLipschitz embeds into $\R$ and has Nagata dimension 0.
  
    \item $\Lam(X)$ has Nagata dimension 0.
    
    \item $\Lam(X)$ biLipschitz embeds into a Hilbert space.
  
    \item $\Lam(X)$ is Markov 2-convex.
    
    \item $\Lam(X)$ does not contain the binary trees with uniformly bounded biLipschitz distortion.
\end{enumerate}
\end{theorem}

\subsection{The key ingredient}
The key conceptual tool  we use throughout the article is a new semimetric on $\Lam(X) = \P_{<\omega}(X) \times X$ (for a set $X$, we denote by $\P(X)$ the power set of $X$ and by $\P_{<\omega}(X)$ the collection of all finite subsets of $X$) which, in general, is significantly less complex than the standard traveling salesman semimetric $\TSP$. The new semimetric, denoted $\TSCP: \Lam(X)\times \Lam(X) \to [0,\infty)$, is defined to be the length of the solution of the {\it ``traveling salesman shortcut problem''}, that is, the shortest distance a salesman travels if they are allowed to shortcut the full problem by visiting only the two points furthest away from each other in $\{x,y\} \cup (A\Delta B)$. Formally, 
$$\TSCP((A,x),(B,y)) := \diam(\{x,y\}\cup (A\Delta B)).$$

Clearly, in any metric space $X$, $\TSCP\le \TSP$ (in the language of diversities \cite{diversities}, one can see 
$\TSCP$ as the basepointed analogue of the {\it diameter diversity} and $\TSP$ as the basepointed analogue of the {\it Steiner tree diversity}). If the weak reverse inequality $\TSP\le K\cdot\TSCP$ holds in $X$ for some constant $K<\infty$, then we say that $X$ is  {\it $K$-$\TSP$-efficient}. The much simpler form of the 
semimetric $\TSCP$ enables us to prove our key technical result which says that for any metric spaces $X,Y$, a biLipschitz embedding of  $X$ into an $\ell_\infty$-product of $Y$ can be lifted, without increasing the distortion, to a biLipschitz embedding of $(\Lam(X), \TSCP) $ into the $\ell_\infty$-product of $(\Lam(Y), \TSCP)$ (Theorem~\ref{thm:induced}).  The approach of considering  the induced embedding was  inspired by the Matoušek and Rabinovich construction of dominated line metrics  \cite{MR01}, however we use  the $\ell_\infty$-products instead of  the $\ell_1$-products that were used in \cite{MR01}.

Another very important advantage of the semimetric $\TSCP$ is that, unlike $\TSP$, it behaves very well with respect to snowflaking, namely, as is easy to see, it satisfies $(\Lam(X),\TSCP^{1-\delta})=(\Lam(X,d^{1-\delta}),\TSCP)$, which is essential for the proof of $(i)\Rightarrow (ii)$ in Theorem~\ref{thm:mainintro} (see Theorem~\ref{thm:TSP-efficient}).

After an initial version of the present paper was posted on arXiv, Anthony Genevois informed us that in his studies of lamplighter spaces through the lens of diadem products, he considered another simplified lamplighter metric which in certain situations is equivalent to $\TSCP$, see \cite{G22}. We thank him for bringing this connection to our attention.

\subsection{Organization}

In \S\ref{sec:prelim}, we recall the definitions and selected properties of notions used in the paper.
Lamplighter spaces are defined in \S\ref{sec-def-lamp} and the definition and first properties of $\TSP$-efficient metric spaces are in \S\ref{sec:TSPefficient}. In \S\ref{S:tsp-ineff} we prove that Hamming cubes embed with uniformly bounded distortion in lamplighter spaces over $\TSP$-inefficient spaces (Theorem~\ref{thm:TSPinefficientHamming}). The key technical 
Theorem~\ref{thm:induced} is in \S\ref{sec-weak-emb}, and the main characterization result of the paper is Theorem~\ref{thm:main} (= Theorem~\ref{thm:mainintro}), which combines several results from  \S\ref{sec:lamplighter}  is presented in \S\ref{S:main}.

In \S\ref{sec4} we briefly  discuss applications of our results to biLipschitz embeddings of lamplighter spaces on finite dimensional  $\TSP$-efficient spaces, including $\Lam(T)$, where $T$ is a  $\TSP$-efficient tree (Corollaries~\ref{cor:induced1-2} 
and \ref{thm:trees}).

In \S\ref{sec:ultra} we obtain the full characterization of metric spaces $X$ such that $\Lam(X)$ biLipschitz embeds into a Hilbert space, Theorem~\ref{thm:main2}(=Theorem~\ref{thm:main2intro}).

\section{Preliminaries} \label{sec:prelim}

We use standard  definitions and notation, as may be found in \cite{Evans,Hei,
NaorRibe,Ost13}, where we also refer the reader for  undefined terms and more details on relevant notions.

We say that $(X,d_X)$ is a {\it semimetric space} if $d_X: X\times X \to [0,\infty)$ satisfies $d_X(x,x) = 0$, $d_X(x,y) = d_X(y,x)$, and $d_X(x,z) \leq d_X(x,y) + d_X(y,z)$ for all $x,y,z\in X$. If, additionally, $d_X(x,y) > 0$ for all $x \neq y \in X$, then $(X,d_X)$ is a {\it metric space}. When $f: X \to Y$ is a map between semimetric spaces, the Lipschitz constant of $f$ is defined to be the (possibly infinite) quantity $\Lip(f) := \inf\{L \geq 0: \forall x,y\in X, \: d_Y(f(x),f(y)) \leq Ld_X(x,y)\}$. Maps $f$ with $\Lip(f)<\infty$ are called {\it Lipschitz}. 
We say that $f: X \to Y$  is a {\it $D$-biLipschitz embedding}, or $f$ has {\it biLipschitz distortion $D$}, if $f$   is a Lipschitz map and  $d_Y(f(x),f(y)) \geq D^{-1}\Lip(f)d_X(x,y)$ for all $x,y\in X$.

\subsection{$\ell_q$-products and wedge sums of semimetric spaces}

Let $\{(X_\alpha,d_\alpha)\}_{\alpha\in A}$ be an indexed family of semimetric spaces equipped with designated basepoints $p_\alpha \in X_\alpha$. For $q \in [1,\infty]$, their {\it $\ell_q$-product} is the semimetric space $(X,d)$ consisting of elements of the Cartesian product $(x_\alpha)_{\alpha\in A} \in \Pi_{\alpha\in A} X_{\alpha}$ such that $\sum_{\alpha\in A}d_\alpha(x_\alpha,p_\alpha)^q < \infty$ (or $\sup_{\alpha\in A}d_\alpha(x_\alpha,p_\alpha) < \infty$, if $q=\infty$), where 
\begin{equation*}
d((x_\alpha)_\alpha,(y_\alpha)_\alpha) =
\begin{cases} \big(\sum_{\alpha\in A} d_\alpha(x_\alpha,y_\alpha)^q\big)^{1/q} & \text{ if } q<\infty,\\
\sup_{\alpha\in A} d_\alpha(x_\alpha,y_\alpha)  & \text{ if } q=\infty.
\end{cases}
\end{equation*}
Note that, when $A$ is a finite set, the resulting product space is independent of the choice of basepoints, and hence we omit reference to them in such cases. We write the $\ell_q$ product semimetric space as 
$\ell_q(\{X_\alpha\}_{\alpha\in A})$. When there is a semimetric space $(X,d)$ with basepoint $p$ such that 
$(X_\alpha,d_\alpha,p_\alpha) = (X,d,p)$ for all $\alpha \in A$, we write $\ell_q^A(X)$ instead of $\ell_q(\{X_\alpha\}_{\alpha\in A})$. If $|A| = m\in\N$, we write $\ell_q^m(X)$, 
and if $(X,d,p) = (\R,|\cdot|,0)$, we  just write $\ell_q^A$ or $\ell_q^m$, 
 respectively.

The {\it wedge sum} of $\{X_\alpha\}_{\alpha\in A}$ is the subset of $\ell_1(\{X_\alpha\}_{\alpha\in A})$ consisting of those $(x_\alpha)_\alpha$ for which $x_\alpha \neq p_\alpha$ for at most one value of $\alpha\in A$.
 
\subsection{Nagata dimension and $\R$-trees}

Let $n\in\N$ and $\gamma<\infty$. Following \cite{LS}, we say that a semimetric space $X$ has {\it Assouad-Nagata dimension $\leq n$ with constant $\leq \gamma$}, or just {\it Nagata dimension $\leq n$ with constant $\leq \gamma$}, if for all scales $s>0$, there exists a covering $\B$ of $X$ such that
\begin{itemize}
    \item $\diam(B) \leq \gamma s$ for every $B\in\B$ and
    \item for every $A \sbs X$ with $\diam(A) < s$, the cardinality of the set $\{B\in\B: B \cap A \neq \emptyset\}$ is at most $n+1$.
\end{itemize}
Equivalently, for all scales $s>0$, there exists a cover of $X$ consisting of $n+1$ subsets $Y_1, \dots Y_{n+1} \sbs X$ such that, for each $i\in\{1,\dots n+1\}$, if $\{y_j\}_{j=0}^k$ is a sequence of points in $Y_i$ with $d(y_{j-1},y_j) \leq s$ for all $j$, then $d(y_0,y_k) \leq \gamma s$. We say that $X$ has {\it finite Nagata dimension} if it has Nagata dimension $\leq n$ with constant $\leq\gamma$ for some $n\in\N$ and $\gamma<\infty$.

\begin{rem} \label{rmk:Nagata}
The second formulation of  Nagata dimension implies that it is a finitely determined property. That is, if every finite subset of $X$ has Nagata dimension $\leq n$ with constant $\leq\gamma$, then $X$ has Nagata dimension $\leq n$ with constant $\leq\gamma$. Indeed, this can be seen by the following simple compactness argument: Fix a scale $s>0$, and choose, for each finite subset $F\in\P_{<\omega}(X)$, a cover $Y_1^F, \dots Y_{n+1}^F$ of $F$ satisfying the second formulation of Nagata dimension $\leq n$ with constant $\leq\gamma$. We think of each set $Y_i^F$ as an element of the compact topological space $\P(X)$, equipped with the topology of pointwise convergence of characteristic functions. Choose an increasing cofinal function $A\ni \alpha \mapsto F_\alpha \in \P_{<\omega}(X)$ such that the subnet $(Y^{F_\alpha}_i)_{\alpha\in A}$ converges for each $i\in\{1,\dots n+1\}$, and let $Y_i$ denote the limit. Then $Y_1, \dots Y_{n+1}$ is a cover of $X$ witnessing that $X$ has Nagata dimension $\leq n$ with constant $\leq\gamma$.
\end{rem}

Perhaps the most important metric spaces in the theory of Nagata dimension are {\it $\R$-trees}. There are many equivalent ways to define this class of spaces. One way is to say that $X$ is an $\R$-tree if any two points in $X$ are joined by a unique arc, that is, the image of a topological embedding $\gamma: [0,1] \to X$ with $\gamma(0) = x$ and $\gamma(1)=y$,
 and this arc is a geodesic. Clearly, any graph theoretical tree - equipped with any weighted shortest path metric - isometrically embeds into an $\R$-tree. See \cite{Evans} for further background on $\R$-trees.

A natural operation on $\R$-trees are wedge sums. The wedge sum of $\{X_\alpha\}_{\alpha\in A}$ is an $\R$-tree whenever $X_\alpha$ is an $\R$-tree for all $\alpha\in A$. If $X_\alpha$ equals $\R$ for all $\alpha$, then the wedge sum is called an {\it $\R$-star}. 

The following result is due to Lang and Schlichenmaier \cite{LS}.

\begin{theorem} \label{thm:LStrees} 
Let $m\in\N$, $q\in [1,\infty]$, and $T$ be an $\R$-tree. Then the Nagata dimension of $\ell_q^m(T)$ is $\leq m$ with constant $\leq\gamma$, where $\gamma<\infty$ depends only on $q, m$.  Consequently, if $(X,d)$ is a semimetric space such that $(X,d^{1-\eps})$ admits an $L$-biLipschitz embedding into $\ell_q^m(T)$ for some $\eps\in(0,1)$ and $L<\infty$, then the Nagata dimension of $(X,d)$ is $\leq m$ with constant $\leq\gamma$, where $\gamma<\infty$ depends only on $q,m,\eps,L$.
\end{theorem}

\begin{proof}
The first part is essentially \cite[Theorem~3.2]{LS}, except that their theorem is not stated with the quantitative control on the constant $\gamma$ we require. However, the quantitative control does hold, as can be seen by examining the proof of \cite[Theorem~3.2]{LS}.

Alternatively, one can argue that the statement without quantitative control formally implies the same statement with quantitative control. Indeed, suppose towards a contradiction that there is a sequence of $\R$-trees $(T_n)_{n\geq 0}$ such that the least Nagata $m$-dimensional constants $\gamma_n$ of $\ell_q^m(T_n)$ tends to $\infty$. Let $T$ be the wedge sum of $\{T_n\}_n$. Then $T$ is again an $\R$-tree, and thus by \cite[Theorem~3.2]{LS}, $\ell_q^m(T)$ has Nagata dimension $\leq m$. Since $\ell_q^m(T)$ contains each space $\ell_q^m(T_n)$ isometrically, this contradicts $\gamma_n\to\infty$.

The second part follows from \cite[Lemma~2.1]{LS}. That lemma is again stated without quantitative control, but the quantitative control obviously must hold, for example, by the same type of argument as in the previous paragraph.
\end{proof}

\subsection{Doubling metric spaces and Assouad dimension}

Let $K\in\N$. A metric space $X$ is called {\it $K$-doubling} if for every $r>0$, every ball $B \sbs X$ of radius $r$ can be covered by $K$ balls $B_1, \dots, B_K$ of radii $r/2$. It is an elementary consequence of Zorn's Lemma that this covering condition is implied by the statement that every $r/2$-separated subset of $B$ has cardinality at most $K$. One of the most important results about doubling spaces is  the Assouad's embedding theorem \cite{Assouad}, cf. \cite[Theorem~12.2]{Hei}, \cite{NN12}. We recall the precise statement here:

\begin{theorem}[Assouad Embedding Theorem] \label{thm:Assouad}
For every $K\in\N$ and $\de\in(0,1)$, there exist $m=m(K,\delta)\in\N$ and $D=D(K,\de)\in(1,\infty)$ such that for every $K$-doubling metric space $(X,d)$, the snowflake space $(X,d^{1-\de})$ admits a $D$-biLipschitz embedding into the Euclidean space $\ell_\infty^m$. Moreover, if $\delta<\frac{1}{2}$, then $m$ can be taken independent of $\delta$.
\end{theorem}

Let $\alpha \in [0,\infty)$ and $\gamma<\infty$. We say that a metric space $(X,d)$ has {\it sharp Assouad dimension $\leq \alpha$ with constant $\leq \gamma$} if for every $x\in X$ and every $0 < r \leq R < \infty$, if $S \sbs B_R(x)$ is $r$-separated, then $|S| \leq \gamma(\frac{R}{r})^\alpha$. The infimum over all such $\alpha$ is the {\it Assouad dimension} of $X$ (allowing the value $\infty$ if no such $\alpha$ exists). The Assouad dimension of a metric space is always an upper bound for its Nagata dimension \cite{LR15}. It is well-known and easy to see that a metric space has finite Assouad dimension if and only if it is doubling (see, for example, \cite[page~25]{LR15}). 

\subsection{Markov type and Hamming cubes}

There are multiple nonlinear notions of type serving as metric space analogues  of the classical  notion of Rademacher type in normed linear spaces, see \cite{NaorRibe} for more information.  Here we recall the definitions of two of these notions.

We say that a semimetric space $(X,d_X)$ has {\it Markov type $p$}, where  $p\in[1,2]$, if there exists a constant $T<\infty$ so that for every stationary reversible Markov chain $\{Z_t\}_{t\geq 0}$ on a finite state space $\Omega$, every map $f: \Omega \to X$, and every $t \geq 0$, we have 
\begin{equation*}
    \E[d_X(f(Z_t),f(Z_0))^p] \leq T^p\E[d_X(f(Z_1),f(Z_0))^p].
\end{equation*}

We say that $(X,d_X)$ has {\it Enflo type $p$},  where  $p\in[1,2]$, if there exists  $T<\infty$ so that for every finite set $I$ and every map on the power set $\theta: \P(I) \to X$, we have 
\begin{equation}\lb{Enflotype}
    \sum_{A\in\P(I)}d_X(\theta(A),\theta(A^c))^p \leq T^p\sum_{i\in I}\sum_{A\in\P(I)}d_X(\theta(A),\theta(A\Delta\{i\}))^p.
\end{equation}

Let $M \in \N$. The {\it $M$-dimensional Hamming cube} is the power set $\P(\{1,\dots M\})$ equipped with the metric 
$d_H(S,T) = |S\Delta T|$ (where $|V|$ denotes the cardinality of the set $V$), that is,  the Hamming cube has the metric structure of a graph equipped with the (unweighted) shortest path metric, where the edge set consists of pairs $\{S,S\Delta\{n\}\}$ for $S,\{n\} \in \P(\{1,\dots M\})$.

A semimetric space $X$ is said to {\it contain the Hamming cubes with uniformly bounded distortion} if there exists  $C<\infty$ such that for every $M\in\N$, the $M$-dimensional Hamming cube admits a $C$-biLipschitz embedding into $X$.

If a semimetric space $X$ has nontrivial Enflo or Markov type (i.e. it has Enflo type $p$ or Markov type $p$, for some $p>1$), then it cannot contain the Hamming cubes with uniformly bounded distortion. In the case of Enflo type, this is clear from the definition. In the case of Markov type, this is implicit in Ball's proof of \cite[Theorem~2.2]{Ball92}, but is also follows from the result that Markov type $p$ implies Enflo type $p$ \cite[Proposition~1]{NS}.

The following deep result is due to Ding, Lee, and Peres.

\begin{theorem}[{\cite[Theorem~1.6]{DLP}}] \label{thm:DLP}
Every semimetric space of finite Nagata dimension has Markov type 2.
\end{theorem}

\begin{rem}[Semimetrics vs. Metrics]
 \cite[Theorem~1.6]{DLP}  is actually stated for metric spaces and not semimetric spaces. However, the distinction is inconsequential.  Indeed, if $X$ is a semimetric space, we may form the quotient metric space $X/\!\!\sim$ by identifing  points at distance 0 to each other. It is easy to see that $X$ has Markov type $p$ if and only if $X/\!\!\sim$ has Markov type $p$, and $X$ has Nagata dimension $\leq n$
  if and only if $X/\!\sim$ has Nagata dimension $\leq n$.
  The analogous statement holds also for Enflo type $p$.
\end{rem}

\section{Lamplighter spaces, $\TSP$-efficiency, and snowflake embeddings} \label{sec:lamplighter}

\subsection{General definition and the lamplighter space on $\R$}\lb{sec-def-lamp}

Let $(X,d)$ be a metric space and $\P_{<\omega}(X)$ be the collection of all finite subsets of $X$. Set
 $$\Lam(X) := \P_{<\omega}(X)\times X.$$Define the length of the solution of the {\it traveling salesman problem} $\TSP: \Lam(X)^2 \to [0,\infty)$ by
\begin{equation*}
    \TSP((A,x),(B,y)) := \inf\left\{\sum_{i=1}^k d(x_{i-1},x_i): x_0 = x, x_k = y, A\Delta B \sbs \{x_i\}_{i=0}^k\right\}.
\end{equation*}
The quantity $\TSP((A,x),(B,y))$ should be thought of as the shortest distance needed for a salesman to travel through all the points of $A\De B$, given that they must start at $x$ and end at $y$.

Note that for any $x\in X$, $A\in \P_{<\omega}(X)$, we have
\begin{equation*}
    \TSP\big( (A,x), (A\De \{x\}, x)\big) =0,
\end{equation*}
while, obviously, $(A,x) \ne (A\De \{x\}, x)$. Thus $\TSP$ is not a metric on $\Lam(X)$, while it can be verified to be a semimetric. The {\it traveling salesman metric} $d_{\Lam}$ on $\Lam(X)$ is defined by
\begin{equation} \label{eq:dLamdef}
    d_{\Lam}((A,x),(B,y)) := \TSP((A,x),(B,y)) + |A\Delta B|.
\end{equation}
It is easy to see that $d_{\Lam}$ is a metric.

\begin{rem}[Motivation for Choice of Lamplighter Metric] \label{rem:metric} 
There are two natural metrics on the Cayley graph of $\Lam(\bbZ)=\bbZ_2\wr\bbZ$ arising from two natural sets of generators $S_1=\{t,a\}$ or  $S_2=\{t,at\}$, where $t=(\emptyset,1)$ and $a=(\{0\},0)$. It is clear that the metrics induced by sets $S_1$ and $S_2$ are 4-biLipschitz equivalent to each other, and thus the choice of $S_1$ or $S_2$ does not affect the biLipschitz geometry of $\Lam(\bbZ)$, and both are used frequently in the literature.
For example, Bartholdi, Neuhauser, and Woess \cite{BW05,Woe05,BNW} used  the metric induced by $S_2$ in their seminal characterization of $\Lam(\bbZ)$ as a horocyclic product of trees, and Naor and Peres \cite{NP08} used the metric induced by $S_1$ in their proof that $\Lam(\bbZ_n)$ biLipschitzly embeds into $L_1$.

The choice of the set $S_1$ induces that elements $(A,x)$, $(B,y)$ of the graph $\Lam(\bbZ)$ are connected by an edge if and only if (1) $A=B$ and $x,y$ are connected by an edge in $\bbZ$, or (2) $x=y$ and $A\De B=\{x\}=\{y\}$. This approach naturally generalizes to arbitrary graphs $X$ and the induced shortest path metric on $\Lam(X)$ is given precisely by the metric $d_{\Lam(X)}$ defined in \eqref{eq:dLamdef}. As discussed in \cite[Section~2]{GT24}, in this graph metric, which they call the {\it ``lazy metric''}, the lamplighter moves from $x$ to $y$ in $X$ and stops at each vertex of $A\De B$ in order to switch the light on or off at that vertex.

Genevois and Tessera \cite{GT24}, call the metric that arises as an analog of using $S_2$ as the set of generators, the {\it ``diligent metric''}. In this metric the lamplighter passes through each point of $A\De B$ but does not need to stop in order to modify the state of the lamp. The diligent metric is equal to  
$\max\{\TSP,\ro\}$, where, here and in the sequel, $\ro$ denotes the semimetric on $\Lam(X)$ defined by 
\begin{equation*}\lb{defro}
    \ro((A,x),(B,y)) = \begin{cases}
    1 \ &\text{ if  } \ A \neq B,\\
    0 \ &\text{ if  } \ A = B.
    \end{cases}
\end{equation*}

In the present work,  we are interested in biLipschitz embeddings of lamplighters and their snowflakes into finite 
$\ell_\infty$-products of $\R$-trees. Since, clearly, $(\Lam(X),\ro)$ biLipschitz embeds into an $\R$-tree with $2^{|X|}$ branches (i.e. an $\R$-star), all our results about biLipschitz embeddings of either $(\Lam(X),\TSP)$ or its snowflakes into finite $\ell_\infty$-products of $\R$-trees, are   valid {\it verbatim} for the lamplighter space  $\Lam(X)$ with the metric $\max\{\TSP,\ro\}$, simply by adding one more $\R$-tree to the $\ell_\infty$-product (which does not affect any estimates for the constants involved). If, in addition, $X$ is 1-separated, then $\TSP((A,x),(B,y)) \geq |A\Delta B|-1$, which implies that
 
\begin{equation} \label{eq:metric}
   \max\{\TSP,\ro\} \leq d_{\Lam} \leq 3\max\{\TSP,\ro\}.
\end{equation}
\end{rem}

\bigskip

We recall a fundamental result about biLipschitz embeddings of $\Lam(\bbZ)$ into an $\ell_1$-product of two trees, first proved by Stein and Taback \cite[Theorem~9]{ST13}, as part of their study of the metric structure of lamplighter groups, Diestel-Leader graphs, and  horocyclic products of trees, cf. also \cite{BW05,Woe05,BNW}. See \cite[Proposition~2.2]{BMSZ21} for an elegant presentation of  this embedding (see \cite{OR19} for an adaptation of it for $\Lam(\bbZ_n)$).

\begin{theorem} \label{thm:Lam(Z)} 
The metric space $(\Lam(\Z),d_{\Lam})$ admits a $3$-biLipschitz embedding into the $\ell_1$-product $B_\infty \oplus_1B_\infty$, where $B_\infty$ denotes the infinite 3-regular unweighted tree. 
\end{theorem}

We will use the following slight generalization of Theorem~\ref{thm:Lam(Z)}.

\begin{theorem} \label{thm:Lam(path)} 
There exists an $\R$-tree $T_\R$ such that the semimetric space $(\Lam(\R),\TSP)$ admits a $3$-biLipschitz embedding into the $\ell_1$-product $T_\R \oplus_1 T_\R$.
\end{theorem}

\begin{proof}
Since the metrics $d_{\Lam}$ and $\max\{\TSP,\ro\}$ coincide on the edges of $\Lam(\Z)$ (with respect to the generating set $S_1$) and since $\max\{\TSP,\ro\}\le d_{\Lam}$ on $\Lam(\Z)$ (due to \eqref{eq:metric}), the embedding of $(\Lam(\Z),d_{\Lam})$ into $B_\infty \oplus_1 B_\infty$, defined in the proof of \cite[Proposition~2.2]{BMSZ21} satisfies the same biLipschitz estimates also for the metric $ \max\{\TSP,\ro\}$. Thus, $(\Lam(\Z), \max\{\TSP,\ro\})$ also admits a 3-biLipschitz embedding into $B_\infty \oplus_1 B_\infty$.

Let $T$ be the $\R$-tree corresponding to $B_\infty$ and $\U$ be a nonprincipal ultrafilter on $\N$. First we show that there is a dense subset $D \sbs \R$ such that $(\Lam(D),\TSP)$ admits a 3-biLipschitz embedding into the $\U$-limit of the sequence of pointed spaces $(2^{-n}(T \oplus_1 T))_{n\in\N}$ (equipped with any basepoint), where $rY$ denotes the rescaling of the metric on a metric space $Y$ by a factor of $r$. Denoting the $\U$-limit of a sequence of pointed metric spaces $(Y_n)_{n\in\N}$ by $\U$-$\lim_{n\to\infty}Y_n$, we will show that $(\Lam(D),\TSP)$ admits a 3-biLipschitz embedding into $\U$-$\lim_{n\to\infty} 2^{-n}(T \oplus_1 T) = \U$-$\lim_{n\to\infty} 2^{-n}T \oplus_1 \U$-$\lim_{n\to\infty} 2^{-n}T$. Since the class of metric spaces isometrically embeddable into $\R$-trees is closed under rescalings and ultralimits (this holds because the property of isometric embeddability into an $\R$-tree is determined by a scale-invariant 4-point inequality, see \cite[Lemma~3.12]{Evans}), this will show that $(\Lam(D),\TSP)$ admits a 3-biLipschitz embedding into the $\ell_1$-product of two $\R$-trees.

Take $D$ to be the dyadic rationals $\{j2^{-n}: j\in\Z,n\in\N\}$. For each $(A,x) \in \Lam(D)$, we can find $n_0\in\N$ large enough so that $(2^nA,2^nx) \in \Lam(\Z)$ for all $n\ge n_0$. Then we define a map from $(\Lam(D),\TSP)$ to $\U$-$\lim_{n\to\infty}(\Lam(\Z),2^{-n}\TSP)$ by sending $(A,x)$ to the $\U$-equivalence class of the sequence $(2^nA,2^nx)_n$. It is easy to see that this map is an isometric  embedding. Further, since $\ro$ is bounded, it follows that $\U$-$\lim_{n\to\infty}(\Lam(\Z),2^{-n}\TSP) = \U$-$\lim_{n\to\infty}(\Lam(\Z),2^{-n}(\max\{\TSP,\ro\}))$. By the first paragraph, each space $(\Lam(\Z),2^{-n}(\max\{\TSP,\ro\}))$ admits a 3-biLipschitz embedding into $2^{-n}(T \oplus_1 T)$, and hence $\U$-$\lim_{n\to\infty}(\Lam(\Z),2^{-n}\TSP)$ admits a 3-biLipschitz embedding into $\U$-$\lim_{n\to\infty}2^{-n}(T\oplus_1 T)$. This completes the construction of the embedding for $\Lam(D)$.

Now, to prove that $(\Lam(\R),\TSP)$ 3-biLipschitz embeds into an $\ell_1$-product of two $\R$-trees, it suffices, by the same ultralimit reasoning of the preceding paragraphs, to show that for every $\eps>0$ and every finite subset $F \sbs \R$, the space $(\Lam(F),\TSP)$ admits a $(3+\eps)$-biLipschitz embedding into the $\ell_1$-product of two $\R$-trees. This easily follows from the previous paragraph and the fact that for every $\eps>0$ and $F \sbs \R$ finite, there is a $(1+\eps)$-biLipschitz embedding of $F$ into $D$.
\end{proof}

\subsection{$\TSP$-efficient metric spaces} \label{sec:TSPefficient}
We define a new semimetric on $\Lam(X)$ as the length of the solution of the {\it traveling salesman shortcut problem} $\TSCP: \Lam(X)^2 \to [0,\infty)$ by
\begin{equation*}
    \TSCP((A,x),(B,y)) := \diam(\{x,y\}\cup (A\Delta B)).
\end{equation*}
The quantity $\TSCP((A,x),(B,y))$ can be thought of as the distance a salesman travels if they are allowed to shortcut the full problem by visiting only the two  points furthest  away from each other in $\{x,y\} \cup (A\Delta B)$. It is easy to verify that $\TSCP$ is a semimetric.

Obviously, the triangle inequality of $d$ implies that $\TSP \geq \TSCP$. Given a constant $K<\infty$, we say that $X$ is {\it $K$-$\TSP$-efficient} if the reverse inequality
\begin{equation*}
    \TSP \leq K\cdot\TSCP
\end{equation*}
holds. We say that $X$ is {\it $\TSP$-efficient} if it is $K$-$\TSP$-efficient for some $K<\infty$. Hence, if   $X$ is $K$-$\TSP$-efficient, then $(\Lam(X),\TSP)$ is $K$-biLipschitz equivalent to $(\Lam(X),\TSCP)$.

\begin{example}[$\TSP$-Efficiency of $\R$ and $\R/\Z$]\label{ex:RTSPefficient}
The line $\R$ and the circle $\R/\Z$, equipped with their geodesic metrics, are $2$-$\TSP$-efficient and $3$-$\TSP$-efficient, respectively. Hence, any subsets of these spaces are also $2$- or $3$-$\TSP$-efficient, respectively.
\end{example}

\begin{example}[$\TSP$-Efficiency of Wedge Sums] \label{ex:wedgesumsTSPefficient} 
If $X_1,\dots X_n$ are $K$-$\TSP$-efficient spaces, then their wedge sum is $nK$-$\TSP$-efficient. In particular, if each $X_i$ is a subset of $\R$ or $\R/\Z$, then their wedge sum is $2n$- or $3n$-$\TSP$-efficient, respectively.
\end{example}

Next we show that $\TSP$-efficient spaces have sharp Assouad dimension $\le 1$.

\begin{lemma} \label{lem:TSPdoubling} 
Let $K<\infty$ and let $(X,d)$ be a $K$-$\TSP$-efficient metric space. Then $X$ has sharp Assouad dimension $\leq 1$ with constant $\leq 2K$. In particular, $X$ is $4K$-doubling.
\end{lemma}

\begin{proof}
First note that the second conclusion follows immediately from the first. We now prove the first. Let $x\in X$ and $0<r\leq R < \infty$. Choose an $r$-separated subset $S \sbs B_R(x)$. We will show that $|S| \leq 2K\frac{R}{r}$, which shows that $\gamma \leq 2K$.

Since $X$ is $K$-$\TSP$-efficient and $\diam(S) \leq \diam(B_R(x)) \leq 2R$, there exists a sequence of points $\{x_i\}_{i=0}^k$ such that $ x_0 = x= x_k$, $S \sbs \{x_i\}_{i=0}^k$, and
\begin{equation*}
   \sum_{i=1}^{k} d(x_{i-1},x_{i}) \leq 2KR.
\end{equation*}
Since $S$ is $r$-separated, and $k\ge |S|$, the left-hand sum is lower bounded by $|S|r$. Solving for $|S|$ yields
 $|S| \leq 2K\frac{R}{r}$.
\end{proof}

We will show in Theorem~\ref{thm:main2} that if $X$ has Assouad dimension $< 1$, then $X$ is $\TSP$-efficient. However, there exist metric spaces with sharp Assouad dimension $\le 1$ that are not $\TSP$-efficient. For example, for any $n\in\bbN$, let $\WB_{2,n}$ be the weighted tree obtained by taking the complete binary tree $B_n$ of depth $n$ and assigning the weight $2^{-k}$ to each edge of $B_n$ at distance $k$ from the root. Let $L_n$ denote the set of leaves of $\WB_{2,n}$. It is easy to see that $L_n$ has sharp Assouad dimension $\le 1$. On the other hand, $\TSP((L_n,\text{root}),(\emptyset,\text{root}))\ge 4n$, while $\diam(L_{n})\leq 4$, which implies that $L_{n}$ is not $K$-$\TSP$-efficient for any $K<n$.

\subsection{Hamming cubes in lamplighters on $\TSP$-inefficient  spaces}
\label{S:tsp-ineff}

We show that $\TSP$-inefficiency of $X$ is a sufficient condition for the lamplighter space over $X$ to contain Hamming cubes with uniformly bounded distortion. We will prove later  that this condition is also necessary. 

\begin{theorem} \label{thm:TSPinefficientHamming}
Let $(X,d)$ be a metric space and $K \geq 4$. If $X$ fails to be $K$-$\TSP$-efficient, then the $M$-dimensional Hamming cube 4-biLipschitz embeds into $(\Lam(X),\TSP)$, where $M = \lfloor \frac{K-4}{3} \rfloor + 1$.
\end{theorem}

\begin{proof}
Assume that $X$ fails to be $K$-$\TSP$-efficient. Let $(A,x),(B,y) \in \Lam(X)$ be such that
\begin{equation}\lb{bigTSP}
    \TSP((A,x),(B,y)) > KR,
\end{equation}
where
\begin{equation*}
    R = \diam(\{x,y\}\cup(A\Delta B)).
\end{equation*}

For brevity, if $C\in \P_{<\omega}(X)$, let us write $\TSP(C,x)$ to mean $\TSP((C,x),(\emptyset,x))$. Note that \eqref{bigTSP} implies
\begin{equation*}
    \TSP(A\Delta B,x) > (K-1)R.
\end{equation*}

Let $\{x_i\}_{i=0}^k \sbs X$ such that $x_0 = x = x_k$, $A\Delta B \subseteq \{x_i\}_{i=1}^k$, and $\sum_{i=1}^k d(x_{i-1},x_i) = \TSP(A\Delta B,x)$. Recursively define a finite sequence of stopping times  $\{\tau(0) < \tau(1) < \dots\}\sbs\{0,\dots k\}$ by
\begin{itemize}
    \item $\tau(0) = 0$
    \item For $n \geq 1$, if $\tau(n-1)$ has been defined and  $\tau(n-1)<k$, then,  since $d(x_{i-1},x_i) \leq R$ for all $i$,  the set $J := \{j \in (\tau(n-1),k]: \sum_{i=\tau(n-1)+1}^j d(x_{i-1},x_i) \leq 2R\}$ is nonempty. Define $\tau(n)=\min\{1+\max J, k\}$.
\end{itemize}
Let $N$ be the maximum value of $n$ for which $\tau(n)$ is defined (so that $\tau(N) = k$), and for each $n\in\{1,\dots N\}$, set $C_n := \{x_i\}_{i=\tau(n-1)}^{\tau(n)-1}$. Then the following facts hold:
\begin{itemize}
    \item[$(a)$] $(K-1)R < \TSP(A\Delta B,x) = \sum_{i=1}^k d(x_{i-1},x_i)$
    
    ~\hspace{.6in} $= \sum_{n=1}^N \Big( d(x_{\tau(n)-1},x_{\tau(n)}) + \sum_{i=\tau(n-1)+1}^{\tau(n)-1} d(x_{i-1},x_i)\Big) \leq \sum_{n=1}^N (R + 2R)$ 
    
 ~\hspace{.6in}  $= 3RN$.

    Hence, $N > \frac{K-1}{3}$.
    \item[$(b)$] $\{C_n\}_{n=1}^{N}$ forms a partition of $\{x_0, x_1,\dots, x_k\}$.
    \item[$(c)$]  For all $n\in\{1,\dots N\}$,
    \begin{equation*}
        \TSP(C_n,x) \leq d(x,x_{\tau(n-1)}) + d(x_{\tau(n)-1},x) + \sum_{i=\tau(n-1)+1}^{\tau(n)-1} d(x_{i-1},x_i) \leq 4R.
    \end{equation*}

    \item[$(d)$] By the minimality of the sum $\sum_{i=1}^k d(x_{i-1},x_i)$,  for all $n\in\{1,\dots N-1\}$, we have
\begin{align*}
    \TSP((C_n,x_{\tau(n-1)}),&(\emptyset, x_{\tau(n)-1})) = \sum_{i=\tau(n-1)+1}^{\tau(n)-1} d(x_{i-1},x_i) \\
    &= \sum_{i=\tau(n-1)+1}^{\tau(n)} d(x_{i-1},x_i) - d(x_{\tau(n)-1},x_{\tau(n)})
    &\geq 2R - R = R.
\end{align*}

Therefore, for any set $S\subseteq \{1,\dots N-1\}$, 
  \begin{equation*}
\begin{split}
     \TSP( \bigcup_{n\in S} C_n,x)&\ge  \sum_{n\in S} \TSP( (C_n,x_{\tau(n-1)}),(\emptyset, x_{\tau(n)-1}))\ge |S|R.
\end{split}
    \end{equation*}
\end{itemize}

Define a map $\theta: \P(\{1,\dots N-1\})\to\Lam(X)$ by $\theta(S) = (\cup_{n\in S}C_n,x)$. Then, by items $(b),(c)$, for any $S,\{n\}\in \P(\{1,\dots N-1\})$, we have
\begin{align*}
    \TSP(\theta(S),\theta(S\Delta\{n\})) = \TSP(C_n,x) \leq 4R.
\end{align*}
And, by items $(b),(d)$, for any $S,T \in \P(\{1,\dots N-1\})$, 
\begin{align*}
    \TSP(\theta(S),\theta(T)) = \TSP( \bigcup_{n\in S\Delta T} C_n,x) \geq |S\De T| R.
\end{align*}
Thus, $\theta$ is a 4-biLipschitz embedding of the $M$-dimensional Hamming cube into  $(\Lam(X),\TSP)$, where $M = N-1>\frac{K-4}{3}$ by item $(a)$. Since $M$ is an integer, this implies $M \geq \lfloor \frac{K-4}{3} \rfloor + 1$.
\end{proof}

\subsection{Key technical result}
\label{sec-weak-emb}

The next theorem about induced maps of $\Lam(X)$ into $\ell_\infty$-product spaces is our key technical result which will enable us to prove our main embedding theorems.

Given a set $X$, an indexed family of metric spaces $\{Y_\la\}_{\la\in \La}$ with basepoints  $\{p_\la\}_\la$, and $q\in [1,\infty]$, we say that a map $f = (f_\la)_{\la\in \La}: X \to \ell_q(\{Y_\la\}_{\la\in \La})$ is {\it componentwise injective} if $f_\la:X\to Y_\la$ is injective for every $\la\in\La$.

\begin{theorem} \label{thm:induced}
Let $(X,d_X)$ and $(Y,d_Y)$ be metric spaces, $L<\infty$, $\Lambda$ an indexing set, and $\varphi: X \to \ell_\infty^\Lambda(Y)$ a componentwise injective $L$-biLipschitz embedding. Then the induced map $\varphi_\#: (\Lam(X),\TSCP) \to \ell_\infty^\Lambda((\Lam(Y),\TSCP))$, defined by 
\begin{equation} \label{eq:induced}
   \varphi_\#(A,x) := \big((\varphi_\lambda(A),\varphi_\lambda(x))\big)_{\la\in\La},
\end{equation}
is an $L$-biLipschitz embedding with $\Lip(\varphi_\#) = \Lip(\varphi)$.
\end{theorem}

\begin{proof}
Since $\varphi$ is componentwise injective, we may freely commute images under $\varphi_\lambda$ and set-theoretic operations (like union and symmetric difference) in the ensuing calculations.  First we compute the Lipschitz constant of $\varphi_\#$. For any  $(A,x),(B,y) \in \Lam(X)$ we have:
\begin{equation*}
\begin{split}
    \sup_{\lambda\in\Lambda}\TSCP((\varphi_\lambda(A),\varphi_\lambda(x)),&(\varphi_\lambda(B),\varphi_\lambda(y))) \\
    &= \sup_{\lambda\in\Lambda}\diam(\varphi_\lambda(\{x,y\} \cup (A \Delta B))) \\
    &\leq \sup_{\lambda\in\Lambda}\Lip(\varphi_\lambda)\diam(\{x,y\} \cup (A \Delta B)) \\
    &= \Lip(\varphi)\TSCP((A,x),(B,y)).
\end{split}
\end{equation*}
This proves that $\Lip(\varphi_\#) \leq \Lip(\varphi)$, while the reverse inequality $\Lip(\varphi_\#) \geq \Lip(\varphi)$ holds a fortiori. 

Next we estimate the coLipschitz constant of $\varphi_\#$. Given $(A,x),(B,y) \in \Lam(X)$, 
let $z,z' \in \{x,y\}\cup (A\Delta B)$ such that $d_X(z,z') =\TSCP((A,x),(B,y))$. Then we have
\begin{equation*}
\begin{split}
    \sup_{\lambda\in\Lambda}\TSCP((\varphi_\lambda(A),\varphi_\lambda(x)),&(\varphi_\lambda(B),\varphi_\lambda(y))) \\
    &= \sup_{\lambda\in\Lambda}\diam(\varphi_\lambda(\{x,y\} \cup (A \Delta B))) \\
    &\geq \sup_{\lambda\in\Lambda} d_Y(\varphi_\lambda(z),\varphi_\lambda(z')) \\
    &\geq \sup_{\lambda\in\Lambda} \frac{\Lip(\vf_\la)}{ L} d_X(z,z') \\
    &=  \frac{\Lip(\vf_\#)}{ L}\TSCP((A,x),(B,y)),
\end{split}
\end{equation*}
which ends the proof.
\end{proof}

To apply Theorem~\ref{thm:induced}, we will need the following simple perturbation lemma. We include its standard  proof for completeness.

\begin{lemma} \label{lem:perturbation}
Let $(X,d_X)$ be a finite metric space, $L<\infty$, $m\in\N$, $\|\cdot\|$ a norm on $\R^m$, and $\varphi: X \to (\R^m,\|\cdot\|)$ an $L$-biLipschitz embedding with $ \Lip(\varphi)>0$. Then for every $\eps>0$, there exists a perturbed map $\tilde{\varphi}: X \to (\R^m,\|\cdot\|)$ that is a componentwise injective $(L+\eps)$-biLipschitz embedding.
\end{lemma}

\begin{proof}
Let $r_1 :=\min_{x\neq y\in X} d_X(x,y) > 0$, $r_2 := \min_{i\leq m}\min_{w\neq z\in \varphi_i(X)} |w-z| > 0$, and $u\in \R^m$ be such that $\|u\| \leq 1$ and $0 < |u_i| \leq 1$ for every $i\leq m$. Let $\e>0$, and define 
$\overline{\e}:=\min\{\e,1\}$, $\e':=\overline{\e} \Lip(\varphi)/(L(L+2))$, and $\de=\min\{r_2/2,\e'r_1\}$.
Choose a function $\psi: X \to \R$ such that $\psi$ is injective and $\max_{x\in X}|\psi(x)| < \delta/2$. Define the perturbation $\tilde{\varphi}: X \to (\R^m,\|\cdot\|)$ by
$$\tilde{\varphi}(x) := \varphi(x) + \psi(x)u.$$
Then the following facts hold:

\begin{itemize}
    \item For every $1\le i\le m$ and every $x\neq y\in X$,
    \begin{itemize}
        \item if $\varphi_i(x) = \varphi_i(y)$, then $\tilde{\varphi}_i(x) - \tilde{\varphi}_i(y) = (\psi(x)-\psi(y))u_i \neq 0$, and
        \item since  $\delta \le r_2/2$,  if $\varphi_i(x) \neq \varphi_i(y)$ then $|\tilde{\varphi}_i(x) - \tilde{\varphi}_i(y)| \geq r_2 - \delta > 0$.
    \end{itemize}
  \end{itemize}
     
    \noindent
    Hence, $\tilde{\varphi}$ is componentwise injective. Further, since $\delta \le r_1\eps'$, we have
\begin{itemize}
    \item $\Lip(\tilde{\varphi}) \leq \Lip(\varphi) + \delta/r_1 \le \Lip(\varphi)+\eps'$.
    
    \item For every $x\neq y \in X$, 
    
    $\|\tilde{\varphi}(x)-\tilde{\varphi}(y)\| \geq \frac{\Lip(\varphi)}{L}d_X(x,y) - \delta \ge (\frac{\Lip(\varphi)}{L}-\eps')d_X(x,y)\ge \frac{ \Lip(\varphi)+\eps'}{L+\e}d_X(x,y)$.
\end{itemize}
Thus $\tilde{\varphi}$ is an $(L+\eps)$-biLipschitz embedding.
\end{proof}

\subsection{Embeddings}\lb{S:emb}

As a consequence of Theorem~\ref{thm:induced} we obtain our two main embedding results.

\begin{cor} \label{cor:induced1}
Let $m\in\N$, $L<\infty$, and $(X,d)$ be a 
finite metric space. If $X$ admits an $L$-biLipschitz embedding into $\ell_\infty^m$, then for all $\eps > 0$, $(\Lam(X),\TSCP)$ admits a $(12L+\eps)$-biLipschitz embedding into a product of $2m$ $\R$-trees $\ell_\infty^{2m}(T_{\R})$, where $T_\R$ is the $\R$-tree satisfying the conclusion of Theorem~\ref{thm:Lam(path)}.
\end{cor}

\begin{proof}
Assume that $X$ admits an $L$-biLipschitz embedding into $\ell_\infty^m$. Fix $\e>0$ and let $\e'=\e/12$. By Lemma~\ref{lem:perturbation}, if  there exists an $L$-biLipschitz embedding of $X$ into $\ell_\infty^m$, then there exists a componentwise injective $(L+\e')$-biLipschitz embedding  of $X$  into $\ell_\infty^m$. Thus, by Theorem~\ref{thm:induced}, there exists an $(L+\e')$-biLipschitz embedding of $(\Lam(X),\TSCP)$ into $\ell_\infty^m(\Lam(\R),\TSCP)$. Since $\R$ is $2$-$\TSP$-efficient, this implies that there exists a $2(L+\e')$-biLipschitz embedding of $(\Lam(X),\TSCP)$ into $\ell_\infty^m(\Lam(\R),\TSP)$. By Theorem~\ref{thm:Lam(path)}, $\ell_\infty^m(\Lam(\R),\TSP)$ 6-biLipschitz embeds into $\ell_\infty^{2m}(T_{\R})$, and thus there exists a $12(L+\e')$-biLipschitz embedding of $(\Lam(X),\TSCP)$ into $\ell_\infty^{2m}(T_{\R})$, which ends the proof of the corollary.
\end{proof}

Since every $K$-$\TSP$-efficient metric space is $4K$-doubling, by the Assouad Embedding Theorem~\ref{thm:Assouad}, every snowflake of a $K$-$\TSP$-efficient space biLipschitz embeds into a finite dimensional space. The simple structure of the semimetric $\TSCP$ enables us to use this fact and Corollary~\ref{cor:induced1} to obtain biLipschitz embeddings of the snowflakes of lamplighters on $\TSP$-efficient spaces. Precisely, we have:

\begin{theorem} \label{thm:TSP-efficient}
For every $K<\infty$ and $\de\in(0,1)$ there exist $m=m(K,\delta)\in\N$ and $L=L(K,\de)\in(1,\infty) $ such that for every finite $K$-$\TSP$-efficient metric space $(X,d)$, the snowflake of the lamplighter space $(\Lam(X),\TSP^{1-\delta})$ admits an $L$-biLipschitz embedding  into a product of $m$ $\R$-trees $\ell_\infty^{m}(T_{\R})$, where $T_\R$ is the $\R$-tree satisfying the conclusion of Theorem~\ref{thm:Lam(path)}. Moreover, if $\delta < \frac{1}{2}$, then $m$ can be taken independent of $\delta$.
\end{theorem}

\begin{proof}
Let $K<\infty$, $\delta\in(0,1)$, and $(X,d)$ be a $K$-$\TSP$-efficient metric space. Then the spaces $(\Lam(X),\TSP^{1-\delta})$ and $(\Lam(X),\TSCP^{1-\delta})$ are $K^{1-\de}$-equivalent. Due to the simple form of the formula for the semimetric $\TSCP$, we have that $(\Lam(X),\TSCP^{1-\delta})=(\Lam(X,d^{1-\delta}),\TSCP)$. 

By Lemma~\ref{lem:TSPdoubling}, the metric space $(X,d)$ is $4K$-doubling, and thus, by Assouad's Embedding Theorem~\ref{thm:Assouad}, there exist $\overline{m}=\overline{m}(K,\delta)\in\N$ (independent of $\delta$ if $\delta < \frac{1}{2}$) and $D=D(K,\de)\in(1,\infty)$ such that the snowflake space $(X,d^{1-\de})$ admits a $D$-biLipschitz embedding into the $\overline{m}$-dimensional space $\ell_\infty^{\overline{m}}$. Hence, by Corollary~\ref{cor:induced1}, for all $\eps > 0$, $(\Lam(X,d^{1-\de}),\TSCP)$ admits a $(12D+\eps)$-biLipschitz embedding into a product of $2\overline{m}$ $\R$-trees $\ell_\infty^{2\overline{m}}(T_{\R})$.

By these two paragraphs, $(\Lam(X),\TSP^{1-\delta})$ admits a $K^{1-\delta}(12D+\eps)$-biLipschitz embedding into $\ell_\infty^{2\overline{m}}(T_{\R})$.
\end{proof}

\subsection{The main characterization}\lb{S:main}

We now arrive at the  main result of the article, a corollary of the results that we've established above, which characterizes those metric spaces $X$ whose lamplighter spaces have nontrivial nonlinear type or  good embeddability properties as exactly those that are $\TSP$-efficient.

\begin{theorem} \label{thm:main}
Let $X$ be a metric space. Equip $\Lam(X)$ with the semimetric $\TSP$. Then the following are equivalent.
\begin{enumerate}[label=(\roman*)]
    \item $X$ is $\TSP$-efficient.
    
    \item Every snowflake of $\Lam(X)$ biLipschitz embeds into a finite product of $\R$-trees.
   
    \item Every snowflake of $\Lam(X)$ biLipschitz embeds into a Hilbert space.
   
    \item $\Lam(X)$ has finite Nagata dimension.
    
    \item $\Lam(X)$ has Markov type 2.
   
    \item $\Lam(X)$ has nontrivial Enflo type.
    
    \item $\Lam(X)$ does not contain the Hamming cubes with uniformly bounded biLipschitz distortion.
\end{enumerate}
\end{theorem}

Before beginning the proof, let us observe that each of the properties concerned in $(i)$-$(vi)$ is finitely  determined, and hence, in order to prove the equivalence of $(i)$-$(vi)$ in Theorem~\ref{thm:main}, it suffices to assume that $X$ is finite as long as the constants do not depend on $|X|$. More specifically, note that each of the properties concerned in $(i)$-$(vi)$ has one or more implicit constants (for $(i)$ it is the $\TSP$-efficiency constant $K$, for $(ii)$ it is the biLipschitz distortion $L$ and the number of trees $m$, etc.), and the whole space $X$ or $\Lam(X)$ has the property with a given implicit constant if and only if every finite subset of $X$ or $\Lam(X)$ has the property with the same implicit constant. For the properties of $\TSP$-efficiency, Markov type $2$, and Enflo type 2, finite determination is manifestly true by their definitions. For finite Nagata dimension, finite determination is true by Remark~\ref{rmk:Nagata}. For biLipschitz embeddings into a finite product of $\R$-trees or into a Hilbert space, finite determination holds due to an ultralimit argument: if every finite subset of a semimetric space admits an $L$-biLipschitz embedding into $Y$, where $Y$ is a Hilbert space or a finite product of $\R$-trees, then the whole semimetric space admits an $L$-biLipschitz embedding into an ultralimit of rescalings of $Y$, which is still a Hilbert space or a finite product of $\R$-trees, respectively. Also note that every finite subset of $\Lam(X)$ is contained in $\Lam(Y)$ for some finite subset $Y\sbs X$. Thus, if we wanted to prove, for example, that $\Lam(X)$ having Markov type 2 implies that $(\Lam(X),\TSP^{1-\delta})$ admits a biLipschitz embedding into a finite product of $\R$-trees, it suffices to assume that $X$ is finite and has Markov type $2$ with constant $T<\infty$, and prove that $(\Lam(X),\TSP^{1-\delta})$ admits an $L$-biLipschitz embedding into an $\ell_\infty$-product of $m$ $\R$-trees, where $m\in\N$ and $L<\infty$ depend only on $T$. Analogous statements hold for proving any of the implications in $(i)$-$(vii)$.

\begin{proof}[Proof of Theorem~\ref{thm:main}]
As explained in the preceding paragraph, we may assume that $X$ is finite in the next sentence. The implication $(i)\implies(ii)$ is Theorem~\ref{thm:TSP-efficient}, $(ii)\implies(iii)$ holds since every finite subset of an $\R$-tree isometrically embeds into a graph-theoretic tree, every snowflake of a graph-theoretic tree biLipschitz embeds into $\ell_2$ (with distortion depending only on the snowflake exponent) by \cite[Theorem~5.1]{LMM05}, and a finite product of $\ell_2$'s is biLipschitz equivalent to $\ell_2$, $(ii)\implies(iv)$ is Theorem~\ref{thm:LStrees}, $(iv)\implies(v)$ is Theorem~\ref{thm:DLP}, $(v)\implies(vi)$ is \cite[Proposition~1]{NS}, $(iii)\implies(vi)$ follows from the facts that $(Y,d)$ has Enflo type $(1-\delta)p$ whenever $(Y,d^{1-\delta})$ has Enflo type $p$ and that Hilbert spaces has Enflo type 2 (\cite{Enflo}, see also \cite[(4.31)]{Ost13}).

In the final two implication, we no longer assume that $X$ is finite. First, $(vi)\implies(vii)$ follows directly from the definitions (cf. \cite[inequality (34)]{NaorRibe}), and finally $(vii)\implies(i)$ is Theorem~\ref{thm:TSPinefficientHamming}.
\end{proof}

\begin{rem} \lb{rem:Zsquare}
Since the Assouad dimension of $\Z^2$ is $2>1$, by Lemma~\ref{lem:TSPdoubling}, it is not $\TSP$-efficient. Hence, by Theorem~\ref{thm:main}, $(\Lam(\Z^2),\TSP)$ contains the Hamming cubes with uniformly bounded distortion. Since $\Z^2$ and the Hamming cubes are 1-separated, it is easy to see (using, for example, \eqref{eq:metric}) that $(\Lam(\Z^2),d_{\Lam})$ (which is  the wreath product $\Z_2 \wr \Z^2$) also contains the Hamming cubes with uniformly bounded distortion. To the best of our knowledge, this result is new.
\end{rem}

\section{Finite dimensional  $\TSP$-efficient  metric spaces} \label{sec4}

As mentioned in the Introduction, it is possible that every $\TSP$-efficient metric space $(X,d)$ itself, not just its snowflakes, can be biLipschitz embedded into a finite dimensional space $\R^m$ for some $m\in\N$ that depends only on the constant of $\TSP$-efficiency, see Problem~\ref{problem-TSP-fd}. For such spaces, by Theorem~\ref{thm:induced} and Corollary~\ref{cor:induced1} we obtain:

\begin{cor} \label{cor:induced1-2}
Let $(X,d)$ be a finite metric space that is $K$-$\TSP$-efficient for some $K<\infty$ and such that $X$ admits an $L$-biLipschitz embedding into $\ell_\infty^m$ for some $m\in\N$ and $L<\infty$. Then
\begin{itemize}
\item[(a)]  for all $\eps > 0$, $(\Lam(X),\TSP)$ admits a $(12KL+\eps)$-biLipschitz embedding into an $\ell_\infty$-product of $2m$ $\R$-trees $\ell_\infty^{2m}(T_{\R})$, where $T_\R$ is the $\R$-tree satisfying the conclusion of Theorem~\ref{thm:Lam(path)},  
\item[(b)]
if, in addition, $(X,d)$ is 1-separated, then, for all $\eps > 0$, the metric  lamplighter space $(\Lam(X),d_{\Lam})$ admits a $(36KL+\eps)$-biLipschitz embedding 
\begin{itemize}
\item[(b1)] into a product of $2m$ trees $\ell_\infty^{2m}(B_\infty)$, where $B_\infty$ denotes the infinite 3-regular unweighted tree,
\item[(b2)]  into every nonsuperreflexive Banach space.
\end{itemize}
\end{itemize}
\end{cor}

\begin{proof}
Since  $X$ is $K$-$\TSP$-efficient, $(\Lam(X),\TSP)$ is $K$-biLipschitz equivalent to  $(\Lam(X),\TSCP)$, and thus part (a) follows immediately from Corollary~\ref{cor:induced1}.

Part (b1) follows by essentially the same reasoning as the proof of Corollary~\ref{cor:induced1}, with an additional observation that when $X$ is 1-separated then, by Lemma~\ref{lem:perturbation} and a standard argument, if  there exists an $L$-biLipschitz embedding of $X$ into $\ell_\infty^m$, then there exists a componentwise injective $(L+\e')$-biLipschitz embedding of $X$ into $\ell_\infty^m(\Z)$. Thus part (c) follows, by Theorems~\ref{thm:induced}, \ref{thm:Lam(Z)}, and since by \eqref{eq:metric}, $(\Lam(X),d_{\Lam})$ is $3$-biLipschitz equivalent to $(\Lam(X),\max\{\TSP,\ro\})$.

Implication (b1)$\Rightarrow$(b2) is well-known. In 2014 Ostrovskii proved, using the Mazur's method of constructing
basic sequences (see \cite[pp.~4-5]{LT77}) , that  finite $\ell_\infty$-products of trees embed into all non-superreflexive Banach spaces with uniformly bounded biLipschitz distortion  (that depends only on the number of trees), \cite[pp.157-159, Step 1 of the proof of Theorem~1.7(a)]{Ost14}.  See also  \cite[Lemma~4.4]{BMSZ21} for an exposition of this proof.
\end{proof}

Corollary~\ref{cor:induced1-2} can be applied for example to $\TSP$-efficient spaces described in Examples~\ref{ex:RTSPefficient} and \ref{ex:wedgesumsTSPefficient}, that is when $X$ is a circle $\R/\Z$, a cycle $\Z_n:=\Z/n\Z$, or a wedge sum of $n$-copies of $\Z_n$ or of $\R$, which gives a new short proof of \cite[Theorem~1.3]{OR19} and 
\cite[Theorem~1.2]{BMSZ21}.

Further, Corollary~\ref{cor:induced1-2} can be applied when $X$ is a $\TSP$-efficient tree. Indeed, by  Lemma~\ref{lem:TSPdoubling}, every $\TSP$-efficient tree $T$ is doubling, and thus, by a result of Gupta, Krauthgamer, and Lee \cite[Theorem~2.5]{GKL03} (see also \cite[Theorem~2.12]{LNP09}, \cite[Theorem~1.2]{DE-BV23}) it embeds into a finite dimensional Euclidean space, where the distortion and the dimension depend only on the doubling constant. Hence, we have:

\begin{cor} \label{thm:trees}
Let $K<\infty$. Then there exist $m\in\N$ and $L<\infty$, depending only on $K$, such that for any  $K$-$\TSP$-efficient 1-separated tree $T$, the lamplighter space $(\Lam(T),d_{\Lam})$ admits an $L$-biLipschitz embedding into $\ell_\infty^m(B_\infty)$, where $B_\infty$ denotes the infinite 3-regular unweighted tree. 
\end{cor}

In particular this answers a question from \cite[p.225, comments after Theorem~1.8]{MR26}. 

\section{Characterization of lamplighter spaces that embed into Hilbert spaces} \lb{sec:ultra}

In this section we characterize metric spaces $X$ such that the lamplighter space $\Lam(X)$ biLipschitz embeds into a Hilbert space. We start by recalling relevant definitions.

A metric space $(X,d)$ is called an {\it ultrametric space} if for all $x, y, z\in X$ we have
\[d(x,y)\le \max\{d(x,z),d(z,y)\}.\]

The following fact is well-known and easy to see.
\begin{fact}\label{ultra=ND0}
Any ultrametric space has Nagata dimension 0 with constant $\gamma=1$, and, conversely, any metric space with Nagata dimension 0 with constant $\leq \gamma<\infty$ is $\gamma'$-biLipschitz equivalent to an ultrametric space, where 
$\gamma'<\infty$ depends only on $\gamma$. 
\end{fact}

We say that a semimetric space  $(X,d)$ is {\it Markov $p$-convex}, where $p>0$, if there exists  a constant $\Pi_p$ such that for every Markov chain 
$\{X_t\}_{t=0}^\infty$ on a state space $\Om$, and every $f:\Om\to X$, we have for every $m\in \N$
\[\sum_{k=0}^m\sum_{t=1}^{2^m} \frac{\E[d(f(X_t),f(\widetilde{X}_t(t-2^k)))^p]}{2^{kp}}\le
 \Pi_p^p\sum_{t=1}^{2^m} \E[d(f(X_t),f(X_{t-1}))^{p}],
\]
where, for any integer $l\ge 0$, $\{\widetilde{X}_t(l)\}_{t=0}^\infty$ denotes the process which equals $X_t$ for time $t\le l$ and evolves independently (with respect to the same transition probabilities) for time $t>l$. The notion of  Markov $p$-convexity was introduced in \cite{LNP09}, where it was also shown that an $\R$-tree $T$ biLipschitzly embeds into a Hilbert space \wtw\ it is Markov 2-convex.

We say that a metric space $X$ {\it contains the paths with uniformly bounded biLipschitz distortion} if there exist $L<\infty$ such that for every $k\in\N$, there exists an $L$-biLipschitz embedding $\R \sps \{0,\dots k\} \to X$.

We are now ready to state the main result of this section, which is a repetition of Theorem~\ref{thm:main2intro}.

\begin{theorem} \label{thm:main2}
Let $X$ be a metric space. Equip $\Lam(X)$ with the semimetric $\TSP$.  Then the following are equivalent.
\begin{enumerate}[label=(\roman*)]
    \item $X$ is $\TSP$-efficient and has Nagata dimension 0.
   
    \item $X$ has Assouad dimension $<1$.
   
    \item $X$ biLipschitz embeds into $\R$ and has Nagata dimension 0.
   
    \item $\Lam(X)$ has Nagata dimension 0.
   
    \item $\Lam(X)$ biLipschitz embeds into a Hilbert space.
  
    \item $\Lam(X)$ is Markov 2-convex.
  
    \item $\Lam(X)$ does not contain the binary trees with uniformly bounded biLipschitz distortion.
\end{enumerate}
\end{theorem}

The proof of Theorem~\ref{thm:main2} occurs at the end of this section, following a host of requisite intermediate results.

The following theorem is readily available by combining well-known results in the literature. It establishes $(ii) \iff (iii)$ in Theorem~\ref{thm:main2}.

\begin{theorem} \label{thm:Assouad<1}
Let $X$ be a metric space. Then the Assouad dimension of $X$ is $<1$ if and only if $X$ biLipschitz embeds into $\R$ and has Nagata dimension 0.
\end{theorem}

\begin{proof}
Assume that the Assouad dimension of $X$ is $<1$. Then since Nagata dimension is a nonnegative integer that is bounded above by the Assouad dimension (\cite[Theorem~1.1]{LR15}), we have that the Nagata dimension of $X$ is 0. Thus, by Fact~\ref{ultra=ND0} and \cite[3.8~Theorem]{ultrametric}, $X$ biLipschitz embeds into $\R$.

Now assume that $X$ biLipschitz embeds into $\R$ and has Nagata dimension 0. For subsets $A$ of $\R$, it is known that Assouad dimension $<1$ is equivalent to Nagata dimension 0. Indeed, by \cite[Remark~6.14]{Guy}, if $A$ has Nagata dimension 0, then $A$ is porous (see \cite[Definition~6.11]{Guy}), if $A$ is porous, then $A$ has Assouad dimension $<1$ \cite{porous}, and if $A$ has Assouad dimension $<1$, then $A$ has Nagata dimension 0 \cite[Theorem~1.1]{LR15}. The desired conclusion follows.
\end{proof}

Let $(X,d)$ be a finite metric space. Define $\TSP(X)$ to be the minimum of $\sum_{i=1}^m d(x_{i-1},x_i)$ among all enumerations $\{x_i\}_{i=0}^m$ of $X$. The next lemma is used to prove Theorem~\ref{thm:TSPultrametricbiLipembedsintoR}, which in turn establishes $(i) \implies (iii)$ in Theorem~\ref{thm:main2}.

\begin{lemma} \label{lem:monotone}
Let $(X,d)$ be a finite, nonempty ultrametric space. Let $\{x_i\}_{i=0}^m$ be an enumeration of $X$ with $\sum_{i=1}^m d(x_{i-1},x_i) = \TSP(X)$. Then the enumeration is monotone in the following sense: 
for every $r> 0$ and $z\in X$, there exist $i_*,i^* \in [0,m]\cap\Z$ such that 
$$\{i \in [0,m]\cap\Z: x_i \in B_r(z)\} = [i_*,i^*] \cap \Z.$$

Moreover, with $r,z,i_*,i^*$ as above, if $B_r(z)$ is $K$-$\TSP$-efficient, then 
$$\sum_{i=i_*+1}^{i^*}d(x_{i-1},x_i) \leq Kr.$$
\end{lemma}

In the proof of Lemma~\ref{lem:monotone} we will use the following elementary facts about ultrametrics:

\begin{fact}\lb{lem:ultra}
Let $(X,d)$ be an ultrametric space, $r> 0$, $z\in X$, $u,u_1,u_2\in B_r(z)$, and $w, w_1,w_2\not\in B_r(z)$. Then
\begin{equation}\lb{in-in}
    d(u_1,u_2)\le r,
\end{equation}
\begin{equation}\lb{out-in}
    d(w,u)> r,
\end{equation}
\begin{equation}\lb{isosceles}
    d(w,u_1)=d(w,u_2),
\end{equation}
\begin{equation}\lb{in-out}
    r+d(w_1,w_2)<d(w_1,u_1)+d(w_2,u_2).
\end{equation}
\end{fact}

\begin{proof}[Proof of Fact~\ref{lem:ultra}]
Inequality \eqref{in-in} follows readily, since $d(u_1,u_2)\le \max\{d(u_1,z),d(z,u_2)\}\le r$.

For \eqref{out-in}, note that $d(w,z)> r$ because $w\not\in B_r(z)$, while $d(u,z)\le r$. Since
$d(w,z)\le \max\{d(w,u),d(u,z)\}$ and $d(u,z)\le r< d(w,z)$, the maximum must be attained at
$d(w,u)$, hence $d(w,u)\ge d(w,z)> r$.

For \eqref{isosceles}, observe that $d(u_1,u_2)\le r< d(w,u_1)$ by \eqref{in-in} and \eqref{out-in}, and therefore
\begin{equation*}
    d(w,u_2)\le  \max\{d(w,u_1),d(u_1,u_2)\}=d(w,u_1).
\end{equation*}
Exchanging the roles of $u_1$ and $u_2$ gives the reverse inequality, and \eqref{isosceles} follows.

Finally, applying the ultrametric inequality and then \eqref{isosceles} to the point $w_2\not\in B_r(z)$, we obtain
\begin{equation*}
    d(w_1,w_2)\le  \max\{d(w_1,u_1),d(u_1,w_2)\}\overset{\eqref{isosceles}}{=}\max\{d(w_1,u_1),d(u_2,w_2)\}.
\end{equation*}
By \eqref{out-in} we also have
\begin{equation*}
    r <\min\{d(w_1,u_1),d(u_2,w_2)\}.
\end{equation*}
Adding together the last two inequalities and using $\min\{a,b\}+\max\{a,b\}=a+b$ yields \eqref{in-out}.
\end{proof}

\begin{proof}[Proof of Lemma~\ref{lem:monotone}]
Assume towards a contradiction that the monotonicity condition fails. Let $r> 0$ and $z\in X$ such that $\{i \in [0,m]\cap\Z: x_i \in B_r(z)\} \neq [i_*,i^*]\cap\Z$ for all $i_*,i^* \in [0,m]\cap\Z$. Then there exist $j< k < \ell \in [0,m]\cap\Z$ such that $x_j,x_\ell \in B_r(z)$ while $x_k \not\in B_r(z)$. Let $[j_*,j^*]\cap\Z, [\ell_*,\ell^*]\cap\Z$ be the largest subintervals of $[0,m]\cap\Z$ such that $j \in [j_*,j^*]$, $\ell \in [\ell_*,\ell^*]$, $\{x_i\}_{i=j_*}^{j^*} \sbs B_r(z)$, and $\{x_i\}_{i=\ell_*}^{\ell^*} \sbs B_r(z)$. We will treat only the case $\ell^* < m$, as the case $\ell^* = m$ is similar and slightly easier.

We reorder the enumeration  $\{x_i\}_{i=0}^m$ by placing the subpath  $\{x_i\}_{i=\ell_*}^{\ell^*}$ in between $x_{j^*}$ and $x_{j^*+1}$, that is, we  consider a new enumeration $\{y_i\}_{i=0}^m$ of $X$ defined by
\begin{equation*}
    y_i := \begin{cases}
            x_i & i \in [0,j^*]\cap\Z \\
            x_{i+\ell_*-j^*-1} & i \in [j^*+1,j^*+1+\ell^*-\ell_*]\cap\Z \\
            x_{i+\ell_*-\ell^*-1} & i \in [j^*+2+\ell^*-\ell_*,\ell^*]\cap\Z \\
            x_{i} & i \in [\ell^*+1,m]\cap\Z \\
           \end{cases}.
\end{equation*}
Then we have
\begin{equation}\lb{eq-mon}
\begin{split}
    \sum_{i=1}^m d(y_{i-1}&,y_i) = \sum_{i=1}^{j^*}  d(x_{i-1},x_i) + d(x_{j^*},x_{\ell_*}) 
    + \sum_{i=j^*+2}^{j^*+1+\ell^*-\ell_*} d(y_{i-1},y_i)
    + d(x_{\ell^*},x_{j^*+1}) \\
    &\hspace{12mm} + \sum_{i=j^*+3+\ell^*-\ell_*}^{\ell^*} d(y_{i-1},y_i)  
    + d(x_{\ell_*-1},x_{\ell^*+1}) + \sum_{i=\ell^*+2}^{m} d(x_{i-1},x_{i}) \\
 & \hspace{3mm}  =\sum_{i\not\in\{j^*+1,\ell_*,\ell^*+1\}} d(x_{i-1},x_i) + d(x_{j^*},x_{\ell_*}) 
    + d(x_{\ell^*},x_{j^*+1})
   + d(x_{\ell_*-1},x_{\ell^*+1}).
   \end{split}
\end{equation}
Since $x_{j^*},x_{\ell_*}, x_{\ell^*}\in B_r(z)$, and $x_{j^*+1},x_{\ell_*-1}, x_{\ell^*+1}\not\in B_r(z)$, by Fact~\ref{lem:ultra}, we have
\begin{equation}\lb{eq1}
 d(x_{\ell^*},x_{j^*+1}) \overset{\eqref{isosceles}}{=} d(x_{j^*},x_{j^*+1}),
\end{equation}

and
\begin{equation}\lb{eq2}
 d(x_{j^*},x_{\ell_*}) +   d(x_{\ell_*-1},x_{\ell^*+1})
 \overset{\eqref{in-in}}{\leq}
 r + d(x_{\ell_*-1},x_{\ell^*+1})
 \overset{\eqref{in-out}}{<}d(x_{\ell_*-1},x_{\ell_*})+  d(x_{\ell^*},x_{\ell^*+1}),
\end{equation}
where \eqref{in-out} is applied with $u_1=x_{\ell_*}$, $u_2=x_{\ell^*}$, $w_1=x_{\ell_*-1}$, and $w_2=x_{\ell^*+1}$.

Plugging \eqref{eq1} and  \eqref{eq2}  into \eqref{eq-mon}, we obtain
\begin{equation*}\lb{eq-mono}
\begin{split}
    \sum_{i=1}^m d(y_{i-1},y_i) &<\sum_{i\not\in\{j^*+1,\ell_*,\ell^*+1\}} d(x_{i-1},x_i) + d(x_{j^*},x_{j^*+1})+ d(x_{\ell_*-1},x_{\ell_*})+  d(x_{\ell^*},x_{\ell^*+1})\\
&= \sum_{i=1}^m d(x_{i-1},x_i) = \TSP(X),  
  \end{split}
\end{equation*}
 which is a contradiction, and 
thus completes the proof of the first part.

For the second part, suppose that $B_r(z)$ is $K$-$\TSP$-efficient. Assume, towards a contradiction, that $\sum_{i=i_*+1}^{i^*}d(x_{i-1},x_i) > Kr$. Then there exists an enumeration $\{w_i\}_{i=i_*}^{i^*}$ of $B_r(z)$ such that $w_{i_*} = x_{i_*}$, $w_{i^*} = x_{i^*}$, and $\sum_{i=i_*+1}^{i^*}d(w_{i-1},w_i) = \TSP(x_{i_*},B_r(z),x_{i^*}) \leq K\diam(B_r(z)) \leq Kr$. Then the enumeration $\{y_i\}_{i=0}^m$ of $X$ defined by
\begin{align*}
    y_i := \begin{cases}
           x_i & i\in[0,i_*]\cap\Z \\
           w_i & i\in[i_*,i^*]\cap\Z \\
           x_i & i\in[i^*,m]\cap\Z
           \end{cases}
\end{align*}
satisfies $\sum_{i=1}^m d(y_{i-1},y_i) < \sum_{i=1}^m d(x_{i-1},x_i)$, a contradiction.
\end{proof}

The next theorem establishes $(i) \implies (iii)$ in Theorem~\ref{thm:main2}.

\begin{theorem} \label{thm:TSPultrametricbiLipembedsintoR}
Let $(X,d)$ be an ultrametric space. Assume that $X$ is $K$-$\TSP$-efficient for some $K<\infty$. Then $X$ admits a $K$-biLipschitz embedding into $\R$.
\end{theorem}

\begin{proof}
Obviously, we may assume that $X\neq\emptyset$. By a compactness argument, it suffices to assume that $X$ is finite. Let $\{x_i\}_{i=0}^m$ be an enumeration of $X$ with $\sum_{i=1}^m d(x_{i-1},x_i) = \TSP(X)$. Define $f: X \to \R$ by $f(x_0)=0$ and for $j\in\{1,\dots,m\}$,
\begin{equation*}
    f(x_j) = \sum_{i=1}^j d(x_{i-1},x_i).
\end{equation*}
Clearly, the triangle inequality implies that $f$ is noncontractive. To see that $\Lip(f) \leq K$, let $j\leq k\in[0,m]\cap\Z$. Fix any $c>1$ and let $\ell \in \Z$ be such that $c^{\ell-1} < d(x_j,x_k) \leq c^{\ell}$, so that $x_k \in B_{c^{\ell}}(x_j)$. Let $i_* \leq i^*\in[0,m]\cap\Z$ be such that $\{i \in [0,m]\cap\Z: x_i \in B_{c^{\ell}}(x_j)\} = [i_*,i^*] \cap \Z$ and $\sum_{i=i_*+1}^{i^*}d(x_{i-1},x_i) \leq Kc^\ell$, which exist by Lemma~\ref{lem:monotone}. Then
\begin{align*}
    |f(x_k)-f(x_j)| = \sum_{i=j+1}^k d(x_{i-1},x_i) \leq \sum_{i=i_*+1}^{i^*} d(x_{i-1},x_i) \leq Kc^\ell \leq cKd(x_j,x_k).
\end{align*}
Since $c>1$ was arbitrary, this implies $\Lip(f) \leq K$.
\end{proof}

The next theorem establishes $(iii) \implies (iv)$ in Theorem~\ref{thm:main2}.

\begin{theorem} \label{thm:La(X)ultrametric}
Let $X \sbs \R$ be such that there exists an ultrametric $d$ for which the identity map $(X,|\cdot|) \to (X,d)$ is an $L$-biLipschitz equivalence. Then $(\Lam(X),\TSCP)$ is $L$-biLipschitz equivalent to an ultrasemimetric space.
\end{theorem}

\begin{proof}
Assume without loss of generality that $L^{-1}|\cdot| \leq d \leq |\cdot|$. Define a new metric $d'$ on $X$ by $d'(w,z) = \sup_{w\leq x\leq y\leq z\in X} d(x,y)$. It is easy to check that $d'$ is still an ultrametric with $L^{-1}|\cdot| \leq d' \leq |\cdot|$ and such that
\begin{equation} \label{eq:d'monotone}
    w\leq x\leq y\leq z \implies d'(x,y) \leq d'(w,z).
\end{equation}
Given $(A,x),(B,y) \in \Lam(X)$, define $m(A,x,B,y) := \min(\{x,y\}\cup(A\Delta B))$ and $M(A,x,B,y) := \max(\{x,y\}\cup(A\Delta B))$. Define $\rho$ on $\Lam(X)$ by
\begin{align*}
    \rho((A,x),(B,y)) := \max\{d'(x,m(A,x,B,y)),d'(y,m(A,x,B,y)),\\
    d'(x,M(A,x,B,y)),d'(y,M(A,x,B,y))\}.
\end{align*}

First, we check that $\rho \leq \TSCP\leq L\rho$. Let $(A,x),(B,y) \in \Lam(X)$. Clearly,
\begin{align*}
    \TSCP((A,x),(B,y)) = |M(A,x,B,y) - m(A,x,B,y)| \geq \rho((A,x),(B,y)).
\end{align*}

Further, since $d'$ is an ultrametric, we have
\begin{align*}
    \TSCP((A,x),(B,y)) &= |M(A,x,B,y) - m(A,x,B,y)| \\
    &\leq Ld'(m(A,x,B,y),M(A,x,B,y)) \\
    &\leq L(d'(x,m(A,x,B,y))\vee d'(x,M(A,x,B,y))) \\
    &\leq L\rho((A,x),(B,y)).
\end{align*}

Finally, we check that $\rho$ is an ultrasemimetric. Let $(A,x),(B,y),(C,z)\in\Lam(X)$. Since $x\in \{x,z\}\cup(A\Delta C)$ and $\{x,z\} \cup (A\Delta C) \sbs \{x,y\} \cup (A\Delta B) \cup \{y,z\} \cup (B\Delta C)$, it holds that
\begin{equation*} \label{eq:isosceles}
    x \geq m(A,x,C,z) \geq m(A,x,B,y) \wedge m(B,y,C,z).
\end{equation*}
This fact together with \eqref{eq:d'monotone} implies that
\begin{align*}
    d'(x,m(A,x,C,z)) \leq d'(x,m(A,x,B,y)) \vee d'(x,m(B,y,C,z)).
\end{align*}
Applying the ultrametric inequality to the last term above yields
\begin{align*}
    d'(x,m(A,x,C,z)) \leq d'(x,m(A,x,B,y)) \vee d'(x,y) \vee d'(y,m(B,y,C,z)),
\end{align*}
and then applying the ultrametric inequality to the middle term above yields
\begin{align*}
    d'(x,m(A,x,C,z)) \leq d'(x,m(A,x,B,y)) \vee d'(y,m(A,x,B,y)) \vee d'(y,m(B,y,C,z)).
\end{align*}
By the definition of $\rho$, the right-hand side is bounded by $\rho((A,x),(B,y)) \vee \rho((B,y),(C,z))$, and hence
\begin{align*}
    d'(x,m(A,x,C,z)) \leq \rho((A,x),(B,y)) \vee \rho((B,y),(C,z)).
\end{align*}

By symmetric arguments, we get that for all $\alpha\in \{d'(z,m(A,x,C,z))$, $d'(x,M(A,x,C,z))$, $d'(z,M(A,x,C,z))\}$,
\begin{align*}
    \alpha \leq \rho((A,x),(B,y)) \vee \rho((B,y),(C,z)).
\end{align*}
Thus, by definition of $\rho$, we have $\rho((A,x),(C,z))$ $\leq \rho((A,x),(B,y)) \vee \rho((B,y),(C,z))$.
\end{proof}

We are now ready to prove Theorem~\ref{thm:main2}. As in the proof of Theorem~\ref{thm:main}, we may assume that the metric space $X$ is finite whenever needed, and prove the quantitative version of the required result under this assumption. We will do this without mention in the proof.

\begin{proof}[Proof of Theorem~\ref{thm:main2}]
The equivalence $(ii) \iff (iii)$ is Theorem~\ref{thm:Assouad<1}. We will prove the equivalences of the remaining items in a cycle skipping $(ii)$.

$(i)\implies(iii)$ follows from Theorem~\ref{thm:TSPultrametricbiLipembedsintoR} and Fact~\ref{ultra=ND0}.

$(iii)\implies(iv)$ follows from Theorem~\ref{thm:La(X)ultrametric}, Fact~\ref{ultra=ND0}, and the fact that subsets of $\R$ are 2-$\TSP$-efficient.

$(iv)\implies(v)$ was proved in \cite{TV79}, see also \cite[\S8.1]{NaorRibe}.

$(v)\implies(vi)$ is \cite[Theorem~3.3]{LNP09}.

$(vi)\implies(vii)$ is \cite[Lemma~3.6 and Claim~3.7]{LNP09}.

For $(vii)\implies(i)$, assume that $\Lam(X)$ does not contain the binary trees with uniformly bounded biLipschitz distortion. Then by \cite{LPP96} (c.f. \cite[Lemma~6.1]{BMSZ21}), $X$ does not contain the paths with uniformly bounded biLipschitz distortion. Then by \cite[Theorem~7.2]{TW05}, $X$ is biLipschitz equivalent to $(Y,d^{1-\eps})$ for some metric space $(Y,d)$ and $\eps>0$. Since $\Lam(X)$ does not contain the binary trees with uniformly bounded biLipschitz distortion, it does not contain the Hamming cubes with uniformly bounded biLipschitz distortion, (cf., e.g., \cite{Bou86}), and thus $X$ is $\TSP$-efficient by Theorem~\ref{thm:main}. Hence, by Lemma~\ref{lem:TSPdoubling}, $X$ has Assouad dimension $\leq 1$. This implies that $(Y,d^{1-\eps})$ has Assouad dimension $\leq 1$, which implies that $(Y,d)$ has Assouad dimension $\leq 1-\eps$ (this is clear from the definition). Since Assouad dimension bounds Nagata dimension (\cite[Theorem~1.1]{LR15}) and since Nagata dimension is a nonnegative integer, this implies that $(Y,d)$ has Nagata dimension 0. Hence, $(Y,d^{1-\eps})$ also has Nagata dimension 0 (this is clear from the definitions, although one may also quote \cite[Theorem~1.2]{LS}), which implies that $X$ has Nagata dimension 0.
\end{proof}

\subsection*{Acknowledgment}
This work was partially supported by the Simons Foundation grant (award no. SFI-MPS-T-Institutes-00010825) and from State Treasury funds as part of a task commissioned by the Minister of Science and Higher Education under the project “Organization of the Simons Semesters at the Banach Center - New Energies in 2026-2028” (agreement no. MNiSW/2025/DAP/491). The authors would like to thank Sylvester Eriksson-Bique and Manisha Garg for their helpful conversations during the Simons Semester at the Banach Center at IMPAN, Warsaw, Poland, which resulted in the proof of Theorem~\ref{thm:Assouad<1}.

\end{document}